\documentclass[leqno,11pt]{amsart}

\usepackage{amssymb}

\renewcommand{\phi}{\varphi}
\renewcommand{\epsilon}{\varepsilon}

\newcommand{\xgeq}{\geqslant}
\newcommand{\xleq}{\leqslant}
\newcommand{\zf}[1]{\,\hspace{0ex}{^{#1}}}
\newcommand{\res}{\!\upharpoonright \! }
\newcommand{\id}{\mathrm{id}}
\newcommand{\continuum}{\mathfrak{c}}
\newcommand{\bbR}{\mathbb{R}}
\newcommand{\bbU}{\mathbb{U}}
\newcommand{\calA}{\mathcal{A}}
\newcommand{\calD}{\mathcal{D}}
\newcommand{\calG}{\mathcal{G}}
\newcommand{\calK}{\mathcal{K}}
\newcommand{\calR}{\mathcal{R}}

\newtheorem{theorem}{Theorem}[section]
\newtheorem{lemma}[theorem]{Lemma}
\newtheorem{corollary}[theorem]{Corollary}

\newcounter{lista}
\newcounter{test}

\newcommand{\test}{\refstepcounter{test}\item[(\arabic{test})]}

\newenvironment{lista}[1]{\begin{list}{\upshape{#1}}{\usecounter{lista}\setlength{\rightmargin}{\leftmargin}}}{\end{list}\setcounter{lista}{0}}

\begin{document}

\title[An example of a rigid $\kappa$-superuniversal metric space]{An example of a rigid $\kappa$-superuniversal metric space}

\author[W. Bielas]{Wojciech Bielas}
\address{University of Silesia, Institute of Mathematics,
Bankowa 14, 40-007 Katowice, POLAND}
\email{wbielas@us.edu.pl}

\date{}

\begin{abstract}
For a cardinal $\kappa>\omega$ a metric space $X$ is called to be $\kappa$-su\-per\-uni\-ver\-sal whenever for every metric space $Y$ with $|Y|< \kappa$ every partial isometry from a subset of $Y$ into $X$ can be extended over the whole space $Y$.
Examples of such spaces were given by Hechler [1] and  Kat\v{e}tov [2].
In particular, Kat\v{e}tov showed that if  $\omega<\kappa=\kappa^{<\kappa}$, then there exists a $\kappa$-superuniversal $K$ which is moreover $\kappa$-homogeneous, i.e.\@ every isometry of a subspace $Y\subseteq K$ with $|Y|<\kappa$ can be extended to an isometry of the whole $K$.
In connection of this W.\@ Kubi\'s suggested that there should also exist a $\kappa$-superuniversal space that is not $\kappa$-homogeneous.
In this paper there is shown that for every cardinal $\kappa$ there exists a $\kappa$-superuniversal space which is rigid, i.e.\@ has exactly one isometry, namely the identity.
The construction involves an amalgamation-like property of a family of metric spaces.
\end{abstract}

\maketitle

\section{Introduction}

\setcounter{test}{0}

If for every pair $Y_0\subseteq Y$ of finite metric spaces and isometric embedding $f_0:Y_0\to X$ there is an isometric embedding $f:Y\to X$ such that $f\res Y_0=f_0$, then we say that $X$ is \emph{finitely injective} (\cite{Melleray_1}).
One of the first such spaces was constructed by Urysohn~\cite{Urysohn}.
The example of Urysohn is a Polish metric space $\bbU$, universal in the class of all separable metric spaces and $\omega$\emph{-homogeneous}, which means that every isometry between finite subsets of $\bbU$ has an extension to an isometry 	of the whole $\bbU$.
Urysohn also showed that in the case of separable Polish metric spaces, the universality (in the class of all separable metric spaces) together with the $\omega$-homogeneity are equivalent to the finite injectivity.
These notions have natural generalizations for infinite cardinal numbers.
In~\cite{Hechler} Hechler, for an uncountable cardinal $\kappa$, defined a metric space $X$ to be $\kappa$\emph{-superuniversal} if for every metric space $Y$ of cardinality at most $\kappa$, every isometry from some subset of $Y$ of cardinality less than $\kappa$ into $X$ can be extended to an isometry of the whole $Y$ into $X$.
Hechler showed that for every uncountable regular cardinal $\kappa$ there exists a $\kappa$-superuniversal metric space of cardinality $\sum_{\lambda<\kappa}2^\lambda$ and that such a space is unique up to isometry if and only if $\kappa=\sum_{\lambda<\kappa}2^\lambda$.
For an uncountable cardinal $\kappa$ Kat\v{e}tov showed (\cite{Katetov}) that if $\kappa^{<\kappa}=\kappa$, then up to isometry there is exactly one $\kappa$-homogeneous metric space of weight $\kappa$, universal in the class of all metric spaces of weight $\kappa$.
Every such space is $\kappa$-superuniversal.
If $\kappa>\omega$, then every $\kappa$-superuniversal metric space is complete.
In~\cite{Katetov} Kat\v{e}tov showed that there is an $\omega$-superuniversal and $\omega$-homogeneous metric space of weight $\omega$ which is also meager.
By the back-and-forth argument, every $\kappa$-superuniversal metric space of cardinality $\kappa$ is $\kappa$-homogeneous and universal in the class of all metric spaces of cardinality less than or equal to $\kappa$.
As Kat\v{e}tov noted in~\cite{Katetov}, the existence of a $\kappa$-superuniversal metric space of weight $\kappa$ for an uncountable $\kappa=\kappa^{<\kappa}$ can be deduced from J\'onsson's theorems concerning relational systems~(\cite{Jonsson}).
By \emph{rigid metric space} we mean a metric space which has no nontrivial isometry~(\cite{Janos}).
Obviously, Kat\v{e}tov's example is not rigid.

The main result of the paper is a construction of a $\kappa$-superuniversal metric space which is rigid.

We also prove that for every $\kappa$ and every $\lambda>\kappa$ which is strongly inaccessible there exists a rigid $\kappa$-superuniversal metric space of cardinality $\lambda$.

\section{Amalgamation of metric spaces}

It is easy to define metrics $d_{01}$, $d_{02}$, $d_{12}$ on sets $\{0,1\}$, $\{0,2\}$, $\{1,2\}$, respectively, in such a way that there is no metric $d$ on $\{0,1,2\}$ such that $(\{i,j\},d_{ij})$ is a subspace of $(\{0,1,2\},d)$ for every $i<j\xleq 2$.
In this case points $0,1,2$ can be viewed as vertices constituting a cycle in a graph, which is also an induced cycle, contained in none of spaces $\{i,j\}$.
We introduce a notion of a graph of a family of metric spaces and prove that if metrics are compatible on intersections of their domains and every induced cycle is contained in some element of the family, then there is a common extension of all the members of the family.
We extensively use this method of amalgamation of a family of metric spaces in order to obtain a rigid $\kappa$-superuniversal  metric space.

We follow standard notions of~\cite{Engelking}; for a pseudometric space $(X,d)$ and $Y\subseteq X$ we set $d\res Y=d\res (Y\times Y)$.
We also consider $Y$ to be a pseudometric space with the pseudometric $d\res Y$ and we will say that $X$ is an \emph{extension of} $Y$.

Assume that for every $s\in S$, $d_s$ is a pseudometric on $X_s$.
Let $G$ be a graph with $\bigcup_{s\in S}X_s$ being the set of vertices where $xy$ is an edge in $G$ if there is $s\in S$ such that $x,y\in X_s$.
We call a sequence $z_0z_1,z_1z_2,\ldots,z_{n-1}z_n$ a \emph{path from} $x$ \emph{to} $y$ if $z_0=x$, $z_n=y$ and $z_iz_{i+1}$ is an edge in $G$ for all $i<n$.
A path $z_0z_1,z_1z_2,\ldots,z_{n-1}z_n$ will be denoted also by $z_0\ldots z_n$.

We say that a graph is \emph{connected} if for its every two points $x$ and $y$ there exists a path from $x$ to $y$.
If a path is of the form $z_0\ldots z_nz_0$, then we say that it is a \emph{closed path}.
If $z_0\ldots z_n$ is a closed path and $z_i\neq z_j$ for all $i<j<n$, then we say that $z_0\ldots z_n$ is a \emph{cycle}.
We say that a path in $G$ is \emph{induced} if no two vertices of the path are connected by an edge that does not itself belong to the path.
Observe that if $d_s\res (X_s\cap X_t)=d_t(X_s\cap X_t)$ for every $s,t\in S$, then the following notion is correctly defined:  if $z_0\ldots z_n$ is a path and $s_1,\ldots,s_n$ are such that $z_iz_{i+1}\in X_{s_i}$ for $i<n$, then we call the number
\[
w(z_0\ldots z_n)=\sum_{i=0}^{n-1}d_{s_{i+1}}(z_i,z_{i+1})
\]
the \emph{weight of} $z_0\ldots z_n$.

\begin{lemma}\label{lem:redukcja_sciezki}
Let $G$ be the graph of the family $\{(X_s,d_s):s\in S\}$ of metric spaces such that $d_s\res (X_s\cap X_t)=d_t\res (X_s\cap X_t)$ for all $s,t\in S$, $s\neq t$.
 If $z_0\ldots z_n$ is a path in $G$ and $z_0\neq z_n$, then there exists a subset $\{z_{i_0},z_{i_1},\ldots,z_{i_m}\}$ such that $z_{i_0}=z_0$, $z_{i_m}=z_n$, $z_{i_0}\ldots z_{i_m}$ is a path and $z_{i_k}\neq z_{i_j}$ for $k<j\xleq m$.
 Moreover, $w(z_{i_0}\ldots z_{i_m})\xleq w(z_0\ldots z_n)$.
\end{lemma}

\begin{proof}
Observe that for $n=1$ our claim holds.

Fix $n<\omega$ and assume that our claim holds for every path $z_0\ldots z_k$ for all $k\xleq n$.
 Fix a path $z_0\ldots z_nz_{n+1}$ and assume that $z_0\neq z_{n+1}$ and there are $i,j\xleq n+1$ such that $z_i=z_j$.
 Consider the case $z_0=z_n$.
 Then $\{z_0,z_{n+1}\}$ is a subset of distinct vertices such that $z_0\ldots z_{n+1}$ is a path and
 \begin{multline*}
w(z_0z_{n+1})=w(z_nz_{n+1})=d_{s_{n+1}}(z_n,z_{n+1})\xleq\\
\sum_{i=0}^nd_{s_{i+1}}(z_i,z_{i+1})=w(z_0\ldots z_nz_{n+1})
 \end{multline*}
 for some $s_1,\ldots,s_{n+1}\in S$ such that $z_i,z_{i+1}\in X_{s_{i+1}}$ for all $i\xleq n$.

 Consider the case $z_0\neq z_n$.
If vertices of the path $z_0\ldots z_n$ are not distinct, then from the induction hypothesis there exists $\{z_{i_0},\ldots,z_{i_m}\}\subseteq\{z_0\ldots z_n\}$ such that $z_{i_0}=z_0$, $z_{i_m}=z_n$, vertices of $z_{i_0}\ldots z_{i_m}$ are  distinct and $w(z_{i_0}\ldots z_{i_m})\xleq w(z_0\ldots z_n)$.
Observe that $m<n$, hence $z_{i_0}\ldots z_{i_m}z_{n+1}$ can be written as $w_0\ldots w_k$ for some $k\xleq n$, thus from the induction hypothesis there exists $\{z_{j_0},\ldots,z_{j_k}\}\subseteq \{z_{i_0},\ldots,z_{i_m},z_{n+1}\}$ such that $z_{j_0}=z_{i_0}$, $z_{j_k}=z_{n+1}$, vertices of $z_{j_0}\ldots z_{j_k}$ are  distinct and $w(z_{j_0}\ldots z_{j_k})\xleq w(z_{i_0}\ldots z_{i_m}z_{n+1})$.
Thus
\begin{multline*}
 w(z_{j_0}\ldots z_{j_k})\xleq w(z_{i_0}\ldots z_{i_m}z_{n+1})=w(z_{i_0}\ldots z_{i_m})+d_{s}(z_{i_m},z_{n+1})\xleq\\
 w(z_0\ldots z_n)+d_{s}(z_{i_m},z_{n+1})=w(z_0\ldots z_{n+1})
\end{multline*}
for some $s\in S$ such that $z_{i_m},z_{n+1}\in X_s$.
If vertices of $z_0\ldots z_n$ are  distinct, then vertices of $z_1\ldots z_{n+1}$ are not distinct and we proceed analogously to the previous setting.
\end{proof}

If $d$ is a pseudometric on $X$, then the graph of a family $\{(X,d)\}$ is a complete graph, hence $xz_0\ldots z_ny$ is a path from $x$ to $y$ for any $x,y,z_0,\ldots,z_n\in X$.
Thus we can define
\[
d(x,Z_1,\ldots,Z_n,y)=\inf\{w(xz_1\ldots z_ny):(z_1,\ldots,z_n)\in Z_1\times\ldots\times Z_n\},
\]
\[
 d(x,Z)=\inf\{d(x,z):z\in Z\},
\]
for $x,y\in X$ and $Z,Z_1,\ldots,Z_n\subseteq X$.

The fact that the class of all metric spaces has the amalgamation property is attributed, by Morley and Vaught, to Sierpi\'nski~\cite{Hung}.
The following theorem gives the amalgam of an arbitrary large family of pseudometric spaces.

\begin{theorem}[Amalgamation lemma]\label{thm:pseudometric_graph}
Assume that $\{(X_s,d_s):s\in S\}$ is a family of pseudometric spaces with connected graph $G$ such that
\begin{lista}{(\roman{lista})}
 \item $d_s\res (X_s\cap X_t)=d_t\res(X_s\cap X_t)$ for all $s,t\in S$,
 \item if $x_1\ldots x_n$ is an induced cycle in $G$, then there is $s\in S$ such that $x_1,\ldots,x_n\in X_s$.
\end{lista}
Then there is a pseudometric $\rho$ on the set $\bigcup_{s\in S}X_s$ such that $\rho\res X_s=d_s$ for every $s\in S$.

Moreover, if there is $s_0\in S$ such that $X_s\cap X_{s_0}\neq\emptyset$ and $X_s\cap X_{t}\subseteq X_{s_0}$ for all $s\neq t$, then for all $s\neq t$, $x\in X_s$, $y\in X_t$, $w\in X_{s_0}\cup X_t$:
  \[
   \rho(x,y)=\rho(x,X_s\cap X_{s_0},X_t\cap X_{s_0},y),
  \]
  \[
   \rho(x,w)=\rho(x,X_s\cap X_{s_0},w).
  \]
\end{theorem}

\begin{proof}
Fix $x,y\in \bigcup_{s\in S}X_s$.
Since $G$ is connected there is a path $z_0\ldots z_n$ from $x$ to $y$, i.e.\@ $z_0=x$, $z_n=y$ and for each $i<n$ there is $s_i\in S$ such that $z_i,z_{i+1}\in X_{s_{i+1}}$.
Thus we define
\[
\rho(x,y)=\inf\{w(z_0\ldots z_n):z_0\ldots z_n\mbox{ is a path from }x\mbox{ to }y\}.
\] 
It is easy to see that $\rho$ is symmetric.
Fix $x,y,z\in\bigcup_{s\in S}X_s$.
Let $y_0\ldots y_n$ be a path from $x$ to $y$, and let $z_0\ldots z_m$ be a path from $y$ to $z$.
Then $y_0\ldots y_nz_1\ldots z_m$ is a path from $x$ to $z$.
Thus $\rho(x,z)\xleq w(y_0\ldots y_nz_1\ldots z_m)$.
Since paths $y_0\ldots y_n$ and $z_0\ldots z_m$ have been taken arbitrarily, we have $\rho(x,z)\xleq \rho(x,y)+\rho(y,z)$.
Obviously $\rho(x,y)\xleq d_s(x,y)$ for $x,y\in X_s$.
Fix $x,y\in X_s$ and a path $z_0z_1z_2$ from $x$ to $y$.
If $z_0=z_2$, then $x=y$ and $d_s(x,y)=0\xleq \rho(x,y)$.
If $z_0\neq z_2$ and $|\{z_0,z_1,z_2\}|\xleq 2$, then $\{z_0,z_1,z_2\}\subseteq X_s$ and
\[
 d_s(x,y)\xleq d_s(x,z_1)+d_s(z_1,y)=w(z_0z_1z_2).
\]
Consider the case $|\{z_0,z_1,z_2\}|=3$.
Since $z_0z_1z_2z_0$ is an induced cycle in $G$, there is $t\in S$ such that $z_0,z_1,z_2\in X_t$.
Then
\[
d_s(x,y)=d_t(x,y)\xleq d_t(z_0,z_1)+d_t(z_1,z_2)=w(z_0z_1z_2).
\]
Fix $2\xleq n<\omega$ and assume that $d_s(x,y)\xleq w(z_0 \ldots z_k)$ for all $2\xleq k< n$, $s\in S$, $x,y\in X_s$ and a path $z_0\ldots z_k$ from $x$ to $y$.
Fix $x,y\in X_s$ and a path $z_0\ldots z_n$ from $x$ to $y$.
By Lemma \ref{lem:redukcja_sciezki} we can assume that $z_0\ldots z_nz_0$ is a cycle.
If $z_0\ldots z_nz_0$ is an induced cycle, then the situation is analogous to the case $n=2$.
Thus assume that $z_0\ldots z_nz_0$ is not an induced cycle.
There is an edge $z_iz_j$ in $G$ which is not an edge in $z_0\ldots z_nz_0$.
We can assume that $i<j\xleq n$, hence $1<i+1<j$  or $i+1<j<n$.
There is $t\in S$ such that $z_i,z_j\in X_t$, hence $z_0\ldots z_{i-1}z_iz_jz_{j+1}\ldots z_n$ is a path from $x$ to $y$ of length $i+1+(n-j)<j+(n-j)=n$, and $z_iz_{i+1}\ldots z_{j-1}z_j$ is a path from $z_i$ to $z_j$ of length $j-i<n$, since $i>0$ or $j<n$.
From the induction hypothesis we have $d_s(x,y)\xleq w(z_0\ldots z_iz_j\ldots z_n)$ and $d_t(z_i,z_j)\xleq w(z_iz_{i+1}\ldots z_{j-1}z_j)$.
Thus
\begin{multline*}
 d_s(x,y)\xleq w(z_0\ldots z_iz_j\ldots z_n)=w(z_0\ldots z_i)+d_t(z_i,z_j)+w(z_j\ldots z_n)\xleq\\
 w(z_0\ldots z_i)+w(z_iz_{i+1}\ldots z_{j-1} z_j)+w(z_j\ldots z_n)=w(z_0\ldots z_n).
\end{multline*}
This shows that $d_s(x,y)\xleq \rho(x,y)$, hence $d_s(x,y)=\rho(x,y)$ for all $x,y\in X_s$.

Assume that $s_0\in S$ is such that $X_s\cap X_{s_0}\neq\emptyset$ and $X_s\cap X_t\subseteq X_{s_0}$ for every $s\neq t$.
Fix $s\neq t$, $x\in X_s$ and $y\in X_t$.
Obviously $\rho(x,y)\xleq \rho(x,X_s\cap X_{s_0},X_t\cap X_{s_0},y)$.
If $x,y\in X_{s_0}$, then $(x,y)\in (X_s\cap X_{s_0})\times (X_t\cap X_{s_0})$, hence
\[
 \rho(x,y)=w(xxyy)\xgeq \rho(x,X_s\cap X_{s_0},X_t\cap X_{s_0},y).
\]
If $x=y$, then $x,y\in X_t\cap X_s\subseteq X_{s_0}$, hence it is the previous case.
Consider the case $x\neq y$ and fix a path $z_0\ldots z_n$ from $x$ to $y$.
By Lemma \ref{lem:redukcja_sciezki} we can assume that $z_0\ldots z_nz_0$ is a cycle.
We will show that there are $i,j\xleq n$ such that $z_i\in X_s\cap X_{s_0}$ and $z_j\in X_t\cap X_{s_0}$.
Consider the case $x\notin X_{s_0}$ and $y\in X_{s_0}$.
Suppose that there is no $i$ such that $z_i\in X_s\cap X_{s_0}$.
Let
\[
 i=\max\{r:z_0,\ldots,z_r\in X_s\setminus X_{s_0}\}.
\]
Since $z_n=y\in X_{s_0}$, we have $i<n$.
Thus $z_{i+1}\in X_s\setminus X_{s_0}$, hence $z_{i+1}\notin X_s$ or $z_{i+1}\in X_{s_0}$.
If $z_{i+1}\in X_s$, then $z_{i+1}\in X_{s_0}$ and $z_{i+1}\in X_s\cap X_{s_0}$, contrary to our assumption.
Thus $z_{i+1}\notin X_s$.
Since $z_iz_{i+1}$ is an edge, there is $p\in S$ such that $z_i,z_{i+1}\in X_p$.
Since $z_{i+1}\notin X_s$, we have $s\neq p$.
Then $X_s\cap X_p\subseteq X_{s_0}$.
From the definition of $i$ we have $z_i\in X_s\setminus X_{s_0}$, hence $z_i\in X_p\cap X_s\subseteq X_{s_0}$, a contradiction.

The case $x\in X_{s_0}$ and $y\notin X_{s_0}$ is analogous to the previous one.

Consider the case $x\notin X_{s_0}$ and $y\notin X_{s_0}$.
Suppose that there are no $i,j\xleq n$ such that $z_i\in X_s\cap X_{s_0}$ and $z_j\in X_t\cap X_{s_0}$.
Then
\begin{itemize}
 \item[$(*)$] $z_i\notin X_s\cap X_{s_0}$ or $z_j\notin X_t\cap X_{s_0}$ for each $i,j\xleq n$.
\end{itemize}
Since $z_0=x\in X_s\setminus X_{s_0}$ and $z_n=y\in X_t\setminus X_{s_0}$, the following numbers are correctly defined:
\[
 i=\max\{r:z_0,\ldots,z_r\in X_s\setminus X_{s_0}\}\text{ and }j=\min\{r:z_r,\ldots,z_n\in X_t\setminus X_{s_0}\}.
\]
Since $X_s\cap X_t\subseteq X_{s_0}$, we have $i<j$.
Thus
\[
 z_{i+1}\notin X_s\setminus X_{s_0}\text{ and }z_{j-1}\notin X_t\setminus X_{s_0}.
\]
By $(*)$ we have
\[
 z_{i+1}\notin X_s\cap X_{s_0}\text{ or }z_{j-1}\notin X_t\cap X_{s_0}.
\]
If $z_{i+1}\notin X_s\cap X_{s_0}$, then $z_{i+1}\notin X_s$.
Since $z_iz_{i+1}$ is an edge, there is $p\in S$ such that $z_i,z_{i+1}\in X_p$.
Since $z_{i+1}\notin X_s$, we have $p\neq s$, hence $X_s\cap X_p\subseteq X_{s_0}$.
By the definition of $i$ we have that $z_i\in X_s\setminus X_{s_0}$, hence $z_i\in X_s\cap X_p\subseteq X_{s_0}$, a contradiction.
If $z_{j-1}\notin X_t\cap X_{s_0}$, then an analogous argument provides a contradiction.
Thus there are $i,j\xleq n$ such that $z_i\in X_s\cap X_{s_0}$ and $z_j\in X_t\cap X_{s_0}$.
Since $z_0\ldots z_n$ is a path there are $s_1,\ldots,s_n\in S$ such that $z_k,z_{k+1}\in X_{s_{k+1}}$ for each $k<n$.
Then
\begin{multline*}
 \rho(x,z_i)+\rho(z_i,z_j)+\rho(z_j,y)\xleq\\ \sum_{k=0}^{i-1}\rho(z_k,z_{k+1})+\sum_{k=i}^{j-1}\rho(z_k,z_{k+1})+\sum_{k=j}^{n-1}\rho(z_k,z_{k+1})=\\
 \sum_{k=0}^{n-1}\rho(z_k,z_{k+1})=\sum_{k=0}^{n-1}d_{s_{k+1}}(z_k,z_{k+1})=w(z_0\ldots z_n).
\end{multline*}
We prove analogously that $\rho(x,w)=\rho(x,X_s\cap X_{s_0},w)$ for every $x\in X_s$ and $w\in X_{s_0}\cup X_t$.
\end{proof}

We obtain the theorem which in particular gives a hedgehog space.

\begin{theorem}[Main theorem]\label{thm:amalgamacja}
Assume that $s_0\in S$ and $\{(X_s,d_s):s\in S\}$ is a family of metric spaces such that for all $s,t\in S$:
\begin{lista}{(\roman{lista})}
\item $X_{s_0}\cap X_s\neq\emptyset$,
\item $X_s\cap X_t\subseteq X_{s_0}$ whenever $s\neq t$,
\item $d_{s_0}\res (X_{s_0}\cap X_s)=d_s\res (X_{s_0}\cap X_s)$.
\end{lista}
Then there exists a metric space $(Y,d)$ such that
\begin{lista}{(\roman{lista})}\addtocounter{lista}{3}
 \item $X_{s_0}$ is a subspace of $Y$,
\item for every $s\in S\setminus\{s_0\}$ there is an isometric embedding $i_s:X_s\to Y$ such that $i_s\res (X_{s_0}\cap X_s)=\id_{X_{s_0}\cap X_s}$, 
\item $Y\subseteq \bigcup_{s\in S}X_s$,
\item for all $s\neq t$, if $x\in i_s[X_s]$, $y\in i_t[X_t]$, $z\in X_{s_0}\cup i_t[X_t]$, then
\[d(x,y)=d(x,X_{s_0}\cap X_s,X_{s_0}\cap X_t,y),\]
\[d(x,z)=d(x,X_{s_0}\cap X_s,z).\]
\end{lista}
Moreover, if there exists $\delta>0$ such that $d_s(x,y)\xgeq \delta$ for every $s\in S\setminus\{s_0\}$, $x\in X_s\setminus X_{s_0}$ and $y\in X_{s_0}$, then $Y=\bigcup_{s\in S}X_s$ and $i_s=\id_{X_s}$ for every $s\in S\setminus \{s_0\}$.
\end{theorem}

\begin{proof}
Suppose that the graph of $\{(X_s,d_s):s\in S\}$ does not satisfy condition (ii) of Theorem \ref{thm:pseudometric_graph}.
Then there exists an induced cycle $z_0\ldots z_nz_0$ such that $\{z_0,\ldots,z_n\}\nsubseteq X_s$ for each $s\in S$.
In particular, $\{z_0,\ldots,z_n\}\nsubseteq X_{s_0}$, hence there exists $i$ such that $z_i\notin X_{s_0}$.
There is $s\in S$ such that $z_i\in X_s$.
Let
\[
 j=\min\{r\xleq i:z_r,\ldots,z_i\in X_s\setminus X_{s_0}\},
\]
\[
 k=\max\{r\xgeq i:z_i,\ldots,z_r\in X_s\setminus X_{s_0}\},
\]
\[
 m=\sup\{r\xgeq 0:z_0,\ldots,z_r\in X_s\setminus X_{s_0}\}.
\]
We will show that if $j>0$, then $z_{j-1}\in X_s\cap X_{s_0}$.
Assume that $j>0$.
Then $z_{j-1}\notin X_s\setminus X_{s_0}$, hence $z_{j-1}\notin X_s$ or $z_{j-1}\in X_{s_0}$.
There exists $t\in S$ such that $z_{j-1},z_j\in X_t$.
If $z_{j-1}\notin X_s$, then $t\neq s$ and $z_j\in X_s\cap X_t\subseteq X_{s_0}$, contrary to the definition of $j$.
Thus $z_{j-1}\in X_s\cap X_{s_0}$.
It can be shown analogously that:
\begin{itemize}
 \item[(a)] if $k<n$, then $z_{k+1}\in X_s\cap X_{s_0}$,
 \item[(b)] if $z_0\in X_s\setminus X_{s_0}$ and $m<n$, then $z_{m+1}\in X_s\cap X_{s_0}$.
\end{itemize}

If $j=0$ and $k=n$, then $\{z_0,\ldots,z_n\}\subseteq X_s$, a contradiction.
Thus $j>0$ or $k<n$.
Consider the case $j>0$ and $k<n$.
Then $z_{j-1},z_{k+1}\in X_s\cap X_{s_0}$.
Thus $z_{j-1}z_{k+1}$ is an edge.
Since $z_0\ldots z_nz_0$ is an induced cycle and $j\xleq k$, we have $j-1=0$ and $k+1=n$, hence $\{z_0,\ldots,z_n\}\subseteq X_s$, a contradiction.
Consider the case $j>0$ and $k=n$.
Then $\{z_{j-1},\ldots,z_n\}\subseteq X_s$.
There is $p\in S$ such that $z_n,z_0\in X_p$.
If $p\neq s$, then $z_k=z_n\in X_s\cap X_p\subseteq X_{s_0}$, contrary to the definition of $k$.
Thus $z_0\in X_s$.
If $z_0\in X_{s_0}$, then $z_0z_{j-1}$ is an edge, hence $j=1$ and $\{z_0,\ldots,z_n\}\subseteq X_s$, a contradiction.
Thus $z_0\in X_s\setminus X_{s_0}$.
Then $m<j-1$ since $\{z_{j-1},\ldots,z_n\}\subseteq X_s$.
By (b), $z_{m+1}\in X_s\cap X_{s_0}$.
Then $z_{m+1}z_n$ is an edge, hence $m+1=n-1$.
Thus $\{z_0,\ldots,z_n\}=\{z_0,\ldots,z_{m+1},z_n\}\subseteq X_s$, a contradiction.
Similar arguments apply to the case $j=0$ and $k<n$.

We have shown that the graph of $\{(X_s,d_s):s\in S\}$ satisfies condition (ii) from Theorem \ref{thm:pseudometric_graph}, thus there is an appropriate pseudometric $\rho$ on the set $\bigcup_{s\in S}X_s$.
It is known that the relation $\sim$, given by the formula
\[
 x\sim y\Leftrightarrow \rho(x,y)=0,
\]
is an equivalence on $\bigcup_{s\in S}X_s$.
Let $Y$ be a transversal of $\{[x]:x\in\bigcup_{s\in S}X_s\}$ such that $X_{s_0}\subseteq Y$, where $[x]$ is an equivalence class of $x$.
Then $(Y,\rho)$ is a metric space.
For $s\in S$ and $x\in X_s$ we define $i_s:X_s\to Y$, $i_s(x)\in[x]\cap Y$.

Assume that there is $\delta>0$ such that $d_s(x,y)\xgeq \delta$ for every $s\in S\setminus\{s_0\}$, $x\in X_s\setminus X_{s_0}$ and $y\in X_{s_0}$.
If $\rho(x,y)=0$, then there is $s\in S$ such that $x,y\in X_s$, since otherwise $\rho(x,y)\xgeq 2\delta$.
Then $0=\rho(x,y)=d_s(x,y)$, hence $x=y$.
\end{proof}
We will call the space $Y$ described above the \emph{amalgam of the family} $\{X_s:s\in S\}$.

\section{The isometric extension}

Fix a metric space $X$ and a cardinal $\kappa$.
We also fix a set $A$ of cardinality $\lambda=(|X|+\continuum)^{<\kappa}$, disjoint with $X$, and a partition $\{A_\alpha:\alpha<\lambda\}\subseteq [A]^\lambda$  of the set $A$.
Let us enumerate (we allow repeating elements)
\[
[X]^{<\kappa}=\{X_\alpha:\alpha<\lambda\}\quad\text{ and }\quad  A_\alpha=\{a_{\alpha,\beta}:\beta<\lambda\}. 
\]
If  $Z$ is a set of cardinality less than $\kappa$, then
\[
|\{\rho\in \zf{Z\times Z}\bbR:\rho\mbox{ is a metric on }Z\}|\xleq \continuum^{<\kappa}\xleq\lambda. 
\]
Thus there exist a family $\{d_{\alpha,\beta}:\beta<\lambda\}$ of metrics on sets from the family $\{X_\alpha\cup \{a_{\alpha,\beta}\}:\beta<\lambda\}$, such that  for every $\alpha<\lambda$:
\begin{itemize}{}
\test $d_{\alpha,\beta}$ is a metric on $X_\alpha\cup\{a_{\alpha,\beta}\}$ for every $\beta<\lambda$,
\test $d\res X_\alpha=d_{\alpha,\beta}\res X_\alpha$ for every $\beta<\lambda$,
\test if $Y$ is {a} metric space and {if} $f_0:Y\setminus\{y\}\to X_\alpha$ is an isometry, then there exists $\beta<\lambda$ and  {an} isometry $f:Y\to X_\alpha\cup\{a_{\alpha,\beta}\}$ such that $f\res (Y\setminus \{y\})=f_0$.\label{wlasnosc_a_alpha_beta}
\end{itemize}
Let $\calR=\{X\}\cup\{X_\alpha\cup\{a_{\alpha,\beta}\}:\alpha,\beta<\lambda\}$.
Observe that with $X_{s_0}=X$ the family $\calR$ satisfies assumptions of Theorem \ref{thm:amalgamacja}, hence there exists an amalgam $(Y,d)$ of $\calR$ such that:
\begin{itemize}{}
 \test $X$ is a subspace of $Y$,
 \test $Y\subseteq \bigcup\{X_\alpha\cup\{a_{\alpha,\beta}\}:\alpha,\beta<\lambda\}$,
 \test for all $\alpha,\beta<\lambda$ there exists an isometric embedding $i_{\alpha,\beta}:X_\alpha\cup\{a_{\alpha,\beta}\}\to Y$ such that $i_{\alpha,\beta}\res X_\alpha=\id_{X_\alpha}$,
 \test for all $y\in Y$ there exists $\alpha<\lambda$ such that for all $x\in X$:
 \[
  d(y,x)=d(y,X_\alpha,x).
 \]
\end{itemize}
The amalgam of the family $\calR$ will be denoted by $F(X)$.

\begin{theorem}\label{thm:F_extension}
For every metric space $X$, $F(X)$ is an extension of $X$ such that
\begin{lista}{(\roman{lista})}
 \item if $|Y|<\kappa$, then every isometric embedding $f_0:Y\setminus\{y\}\to X$ has an extension to an isometric embedding $f:Y\to F(X)$,
 \item for every $y\in F(X)$ there is $Z\in[X]^{<\kappa}$ such that  $d(x,y)=d(x,Z,y)$ for each $x\in X$, where $d$ is the metric of $F(X)$.
\end{lista}
\end{theorem}

\begin{proof}
Fix a metric space $X$ and assume that $\{X_\alpha:\alpha<\lambda\}=[X]^{<\kappa}\setminus\{\emptyset\}$ and $\{d_{\alpha,\beta}:\beta<\lambda\}$ satisfy conditions (1)--(3) from the beginning of this section.

(i) Fix a metric space $Y$, $|Y|<\kappa$, $y\in Y$, and an isometric embedding $f_0:Y\setminus \{y\}\to X$.
There is $\alpha<\lambda$ such that $f_0[Y\setminus\{y\}]=X_\alpha$.
Thus $f_0:Y\setminus\{y\}\to X_\alpha$ is an isometry.
By (3) there is $\beta<\lambda$ and an isometry $g:Y\to X_\alpha\cup\{a_{\alpha,\beta}\}$ such that $g\res (Y\setminus\{y\})=f_0$.
By (6) there is an isometric embedding $i_{\alpha,\beta}:X_\alpha\cup\{a_{\alpha,\beta}\}\to F(X)$ such that $i_{\alpha,\beta}\res X_\alpha=\id_{X_\alpha}$.
The function $f=i_{\alpha,\beta}\circ g:Y\to F(X)$ is an isometric embedding.
For all $y'\in Y\setminus\{y\}$ we have
\[
 f(y')=i_{\alpha,\beta}(g(y'))=i_{\alpha,\beta}(f_0(y))=f_0(y),
\]
since $f_0(y)\in X_\alpha$.
Thus $f\res (Y\setminus\{y\})=f_0$.

(ii) Fix $y\in F(X)$.
By (7) there is $\alpha<\lambda$ such that $d(y,x)=d(y,X_\alpha,x)$ for all $x\in X$.
\end{proof}

In~\cite{Katetov} Kat\v{e}tov construed an extension $E(X,\kappa)$ which satisfies conditions (i) and (ii) of the above theorem.
The metric $\sigma$ of the extension $E(X,\kappa)$ satisfies for $x,y\in E(X,\kappa)\setminus X$ the following equality
\[
\sigma(x,y)=\sup\{|\sigma(x,z)-\sigma(y,z)|:z\in X\}
\]
that is, the distance $\sigma(x,y)$ is as small as possible.
It can be shown that for $x,y\in F(X)\setminus X$ we have
\[
 \rho(x,y)=\inf\{\rho(x,z)+\rho(z,y):z\in X\}
\]
that is, the distance $\rho(x,y)$ is as large as possible.

\section{The discrete character of a point}

We say that a metric space $(Y,\sigma)$ is \emph{discrete} if $\sigma(x,y)\in\{0,1\}$ for all $x,y\in Y$.
If $Y$ is a discrete subspace of $(X,d)$, then by the \emph{middle point of} $Y$ \emph{in} $X$ we mean $x\in X$ such that $d(x,y)=1/2$ for every $y\in Y$.
If there is no middle point of a discrete subspace $Y\subseteq X$, then we say that $Y$ is \emph{without middle points in} $X$.
If cardinal $\kappa$ is fixed and every $Z\in [Y]^\kappa$ is without middle points, then we say that $Y$ is \emph{hereditarily without middle points in} $X$.
We call the cardinal
\begin{multline*}
\tau_\kappa(x,X)=\sup\{|Y|:Y\subseteq X\text{ is hereditarily without middle points in }X\\\
\text{and }x\in Y\} 
\end{multline*}
the \emph{discrete character of }$x$ \emph{in} $X$.

\begin{lemma}\label{lem:isometry_preserves_tau}
 If $f:X\to X$ is an isometry, then $x$ and $f(x)$ have the same discrete character.
\end{lemma}

\begin{proof}
Fix $x\in X$, an isometry $f:X\to X$ and hereditarily without middle points subspace $Y\subseteq X$ such that $x\in Y$.
Let $d$ be the metric of $X$.
Suppose that $f[Y]$ is not hereditarily without middle points.
Thus there is $Z\in [f[Y]]^\kappa$ with a middle point $z\in X$.
Then $f^{-1}[Z]\in [Y]^{\kappa}$. 
Observe that $d(f^{-1}(z),f^{-1}(x))=d(z,x)$ for all $z\in Z$.
Then $f^{-1}(z)$ is a middle point of $f^{-1}[Z]$, a contradiction.
\end{proof}

By the \emph{weak middle point of} a discrete subspace $Y\subseteq X$ we mean $x\in X$ such that $d(x,y)=d(x,y')<1$ for all $y,y'\in Y$.
By $d(x,A)=\inf\{d(x,a):a\in A\}$ we denote the distance between $x$ and $A$ if $A\neq\emptyset$.

\begin{lemma}\label{lem:reduction_of_middle_point}
 Assume that $\kappa$ is a regular cardinal such that $\lambda^{\aleph_0}<\kappa$ for $\lambda<\kappa$.
 Assume that  $(X,d)$ is an extension of  $Y$ such that for every $x\in X\setminus Y$ there is $Z\in [Y]^{<\kappa}$ such that $d(x,y)=d(x,Z,y)$ for all $y\in Y$.
Then for every $\epsilon>0$ and a weak middle point $x\in X\setminus Y$ of a discrete subspace $D\in[Y]^{\kappa}$ there exist $D'\in [D]^\kappa$ and a weak middle point $y\in Y$ of $D'$ such that $d(x,y)+d(y,D')<d(x,D)+\epsilon$.
\end{lemma}

\begin{proof}
We can assume that $\epsilon<1-d(x,D)$.
Let $Z\in [Y]^{<\kappa}$ be such that $d(x,y)=d(x,Z,y)$ for all $y\in Y$.
For each $y\in D$ there is $(t_{y,n})_{n<\omega}\subseteq Z$ such that $d(x,y)=\lim_{n\to\infty}(d(x,t_{y,n})+d(t_{y,n},y))$.
Thus we have function $\Phi:D\to \zf{\omega}(Z\times\bbR)$,
\[
\Phi(y)=(t_{y,n},d(t_{y,n},y))_{n<\omega}
\]
 for $y\in D$. 
Since $|\zf{\omega}(Z\times\bbR)|=|Z|^{\aleph_0}\cdot\continuum<\kappa$, there is $(z_n,r_n)_{n<\omega}\in\zf{\omega}(Z\times\bbR)$ such that $|\Phi^{-1}[\{(z_n,r_n)_{n<\omega}\}]|=|D|$.
Let $D'=\Phi^{-1}[\{(z_n,r_n)_{n<\omega}\}]$.
There is $n<\omega$ such that $d(x,z_n)+d(z_n,D')<d(x,D)+\epsilon<1$.
Fix $y,y'\in D'$.
Then $\Phi(y)=\Phi(y')$, hence $t_{y,n}=t_{y',n}=z_n$ and $d(z_n,y)=d(z_n,y')=r_n<1$.
Thus $z_n\in Y$ is a weak middle point of $D'$.
\end{proof}

\section{The operation of adding discrete subspaces}

Fix a metric space $X$ and $Y\subseteq X$.
Let $X\setminus Y=\{x_\alpha:\alpha<\lambda\}$ for some cardinal number $\lambda$.
For every $\alpha<\lambda$, let $D_\alpha$ be a set disjoint with $X$ such that $|D_\alpha|=\aleph_{\kappa+|X|+\alpha+2}$.
We can assume that $D_\alpha\cap D_\beta=\emptyset$ for $\alpha<\beta<\lambda$.
Let $d_\alpha$ be a metric on the set $D_\alpha\cup\{x_\alpha\}$ which makes it a discrete space.
Using Theorem \ref{thm:amalgamacja} for the family $\{X\}\cup\{D_\alpha\cup\{x_\alpha\}:\alpha<\lambda\}$ we obtain an amalgam  $A(X,Y)=X\cup\bigcup_{\alpha<\lambda}D_\alpha$ and a family $\calA(X,Y)=\{D_\alpha\cup\{x_\alpha\}:\alpha<\lambda\}$ such that
\begin{itemize}
 \item[(A1)]  for every $x\in X\setminus Y$ there exists a discrete subspace $D_x\in\calA(X,Y)$ such that $D_x\cap X=\{x\}$ and $|D_x|>|X|\cdot \kappa^+$,
\item[(A2)] for every $x,y\in X\setminus Y$, if $x\neq y$, then $|D_x|\neq |D_y|$,
\item[(A3)] $d(x',y)=d(x',x)+d(x,y)=1+d(x,y)$ for all $x\in X\setminus Y$, $x'\in D_x\setminus\{x\}$ and $y\in X$, where $d$ is the metric of $A(X,Y)$.
\item[(A4)] $X$ is a subspace of $A(X,Y)$.
\end{itemize}

We define
\[
\calD(X)=\{Y\in [X]^{\kappa}:Y\mbox{ is hereditarily without middle points}\}.
\]
Fix $\calG\subseteq \calD(X)$.
Let $\{Y_\alpha:\alpha<\lambda\}=\{Y\in\calD(X):\forall_{Z\in \calG}|Y\cap Z|<\kappa\}$
for some cardinal number $\lambda$.
Let $Z=\{z_\alpha:\alpha<\lambda\}$ be a set of cardinality $\lambda$, disjoint with $X$.
For every $\alpha<\lambda$ let
\[d_\alpha(x,y)=\left\{\begin{array}{ll}
                   d(x,y),&\mbox{if }x,y\in Y_\alpha,\\
\frac 12,&\mbox{if }y\in Y_\alpha, x=z_\alpha.
                  \end{array}\right.
                  \]
Then $d_\alpha$ is a metric on $Y_\alpha\cup\{z_\alpha\}$.
Using Theorem \ref{thm:amalgamacja} for $\{X\}\cup\{Y_\alpha\cup\{z_\alpha\}:\alpha<\lambda\}$ we obtain an amalgam $S(X,\calG)=X\cup Z$ with metric $\rho$ such that
\begin{itemize}
\item[(S1)] for each $Y\in\calD(X)$ such that $|Y\cap Z|<\kappa$ for all $Z\in\calG$, there is $Y'\in [Y]^\kappa$ and a middle point $x\in S(X,\calG)$ of $Y'$,
\item[(S2)] for every $y\in S(X,\calG)$ there is a discrete subspace $Y\in [X]^{\kappa}$ such that $\rho(y,x)=\rho(y,Y,x)$ for all $x\in S(X,\calG)\setminus\{y\}$.
\end{itemize}

\section{A construction of a rigid $\kappa$-superuniversal metric space}

We recall the following easy observation.

\begin{lemma}\label{lem:granica}
Let $\{Z_\beta:\beta<\alpha\}$ be a family of metric spaces such that $Z_\beta$ is a subspace of $Z_{\gamma}$, for all $\beta<\gamma$.
Then there exists the metric $\rho$ on the set $\bigcup_{\gamma<\alpha}Z_\gamma$ such that $Z_\beta$ is a subspace of $\bigcup_{\gamma<\alpha}Z_\gamma$, for all $\beta<\alpha$.
\end{lemma}

Assume that $\kappa>\continuum$ is a regular cardinal such that $\lambda^{\aleph_0}<\kappa$ for $\lambda<\kappa$, for example $\kappa=\continuum^+$.
We define the empty metric space $X_0=\emptyset$ and a single-element metric space $X_1=\{0\}$.
Assume that we have constructed an increasing chain of metric spaces $\{(X_\beta,d_\beta):\beta<\alpha\}$.

If $\alpha$ is a limit ordinal, then we set
\[
 X_{\alpha+1}=X_\alpha=\bigcup_{\beta<\alpha}X_\beta\quad\quad\text{ and }\quad\quad\calA(X_{\alpha+1},X_\alpha)=\emptyset.
 \]
If $\alpha=\beta+2$ for some $\beta<\alpha$, then we set
\begin{equation}\label{eqn:x_alpha}
 X_\alpha=F(S(A(X_{\beta+1},X_\beta),\bigcup_{\gamma\xleq\beta}\calA(X_{\gamma+1},X_\gamma))).
\end{equation}
Thus we obtain a metric space $(X_{\kappa^+},d)$, $\{X_\alpha:\alpha<\kappa^+\}$ and $\{\calA(X_{\alpha+1},X_\alpha):\alpha<\kappa^+\}$ with the following properties:
\begin{lista}{(P\arabic{lista})}
 \item $X_{\kappa^+}=\bigcup\{X_\alpha:\alpha<\kappa^+\}$,
 \item $X_\beta$ is a subspace of $X_\alpha$ for all $\beta<\alpha<\kappa^+$,
 \item for all $\alpha<\kappa^+$, if $\alpha$ is limit, then $X_{\alpha+1}=X_\alpha=\bigcup\{X_\beta:\beta<\alpha\}$ and $\calA(X_{\alpha+1},X_\alpha)=\emptyset$,
 \item if $\alpha=\beta+2<\kappa^+$, then the formula (\ref{eqn:x_alpha}) holds,
 \item for all $x\in X_{\kappa^+}$ there exists $\alpha<\kappa$ and a discrete subspace $D_x\in\calA(X_{\alpha+1},X_\alpha)$ such that $x\in X_{\alpha+1}\setminus X_\alpha$ and $D_x\cap X_{\alpha+1}=\{x\}$,
 \item every point of $X_{\kappa^+}\setminus X_1$ is added by $F$, $S$ or $A$.
\end{lista}

We define \emph{rank} $r(x)=\min\{\alpha<\kappa^+:x\in X_\alpha\}$ \emph{of} $x\in X_{\kappa^+}$.

\begin{theorem}\label{thm:X_kappa_is_kappa-superuniversal}
 $X_{\kappa^+}$ is $\kappa$-superuniversal.
\end{theorem}

\begin{proof}
Fix a metric space $Y$ of cardinality less than $\kappa$, $y\in Y$ and an isometric embedding $f_0:Y\setminus\{y\}\to X_{\kappa^+}$.
 There exists $\alpha<\kappa^+$ such that $f_0[Y\setminus\{y\}]\subseteq X_{\alpha}$.
 Moreover we have
 \[
X_\alpha\subseteq X_{\alpha+2}\subseteq S(A(X_{\alpha+2},X_{\alpha+1},\bigcup_{\gamma\xleq\alpha}\calA(X_{\gamma+1},X_\gamma))).  
 \]
 From the condition (i) of Theorem \ref{thm:F_extension} there exists an isometric embedding $f:Y\to X_{\alpha+3}$ such that $f\res (Y\setminus \{y\})=f_0$.
 The rest of the proof is an easy induction on the number of points we have to add to the domain of an isometric embedding.
\end{proof}

The proof of the main theorem will be preceded by a series of lemmas.

\begin{lemma}\label{lem:operations_A_and_S_do_not_add_middle_points}
Assume that $x,y\in X_{\kappa^+}$ and $y$ is a weak middle point of $D\in[D_x]^\kappa$.
Then $r(x)<r(y)$ and $y$ is added by $F$ or $S$.
Moreover, if $y$ is a middle point of $D$, then $y$ is added by $F$.
\end{lemma}

\begin{proof}
There are $\beta,\delta<\kappa^+$ such that $x\in X_{\beta+2}\setminus X_{\beta+1}$ and $y\in X_{\delta+2}\setminus X_{\delta+1}$.
Thus $r(y)=\delta+2$ and $r(x)=\beta+2$.

Suppose that $r(x)\xgeq r(y)$.
Since $r(y)\xleq\beta+2$ we have $y\in X_{\beta+2}$, hence by the property (A3):
\[
d(v,y)=d(v,x)+d(x,y)=1+d(x,y)\xgeq 1 
\]
for all $v\in D_x\setminus\{x\}$, a contradiction, since $y$ is a weak middle point of $D\subseteq D_x$.
Thus $r(x)<r(y)$, hence $r(x)\xleq \delta+1$.

Suppose that $y$ is added by $A$.
Then there is $w\in X_{\delta+1}$ such that $y\in D_w\setminus\{w\}$.
Thus $D_x\cup D_w\subseteq A(X_{\delta+1},X_\delta)$,
hence by (A3) we have
\[
 d(y,v)=d(y,w)+d(w,v)\xgeq 1
\]
for all $v\in D_x$, contrary to the fact that $y$ is a weak middle point of $D$.

Suppose that $y$ is a middle point of $D$ added by $S$.
Then there exists a discrete subspace $Z\in [A(X_{\delta+1},X_\delta)]^{\kappa}$ such that
\[
d(y,w)=d(y,Z,w)\quad\text{ for all }w\in A(X_{\delta+1},X_\delta)\setminus\{y\},
\]
and $|Z\cap D'|<\kappa$ for all $D'\in \bigcup_{\zeta\xleq \delta}\calA(X_{\zeta+1},X_{\zeta})$.
Since $r(x)\xleq \delta+1$ we have $x\in X_{\delta+1}$, hence  $D_x\in\bigcup_{\zeta\xleq\delta}\calA(X_{\zeta+1},X_\zeta)$.
Thus $|Z\cap D_x|<\kappa$ and $|Z\cap D|<\kappa$, hence there is $v\in D\setminus Z$.
Then
\[
 \frac 12=d(y,v)=d(y,Z,v)=\frac 12+\inf\{d(z,v):z\in Z\},
\]
 hence $\inf\{d(z,v):z\in Z\}=0$.
 There exists $z_0\in Z$ such that $d(v,z_0)<\frac12$.
 Since $v\notin Z$ and $z_0\in Z$, we have $v\neq z_0$,  $d(v,z_0)>0$.
 Using the formula $\inf\{d(z,v):z\in Z\}=0$ again we obtain $z_1\in Z$ such that $d(z_1,v)<d(z_0,v)$.
 Thus $z_0$ and $z_1$ are distinct elements of $Z$.
Then $1=d(z,z')\xleq d(z,v)+d(z',v)<1$, a contradiction. 
\end{proof}

\begin{lemma}\label{lem:approximation_of_weak_middle_point}
 Assume that $x,y\in X_{\kappa^+}$ and $y$ is a weak middle point of some $D\in[D_x]^{\kappa}$, and $y$ is added by $F$.
 Then for every $\epsilon>0$ there exist $D'\in[D]^\kappa$ and a weak middle point $y'$ of $D'$ such that
 \[
d(y,y')+d(y',D')<d(y,D)+\epsilon                                                                                                           
\]
and
 \begin{lista}{(\roman{lista})}
  \item  $r(y')\xleq r(y)$ and $y'$ is added by $S$, or
  \item $r(y')<r(y)$ and $y'$ is added by $F$.
 \end{lista}
\end{lemma}

\begin{proof}
Since $y$ is added by $F$, there exist $\beta<\kappa^+$ and $Z\in [T]^{<\kappa}$ such that
\[
y\in X_{\beta+2}\setminus T,
\]
where $T=S(A(X_{\beta+1},X_\beta),\bigcup_{\delta\xleq\beta}\calA(X_{\delta+1},X_\delta))$ and $d(y,t)=d(y,Z,t)$ for all $t\in T$.

By Lemma \ref{lem:operations_A_and_S_do_not_add_middle_points} we have $r(x)<r(y)$, hence $D_x\subseteq T$.
By Lemma \ref{lem:reduction_of_middle_point} we obtain a weak middle point $y'\in T$ of some $D'\in[D]^\kappa$ such that $d(y,y')+d(y',D')<d(y,D)+\epsilon$.
Observe that $r(y')\xleq r(y)$, hence if $y'$ is added by $S$, then the proof is complete.

Assume that $y'$ is not added by $S$.
 By Lemma \ref{lem:operations_A_and_S_do_not_add_middle_points} the point  $y'$ is added by $F$, hence it suffices to show that $r(y')<r(y)$.
Each point of $T$ is added by $S$, $A$ or it belongs to $X_{\beta+1}$.
Since $y'\in T$ and $y'$ is added neither by $S$, nor by $A$, we have $y\in X_{\beta+1}$.
Thus $r(y')\xleq \beta+1<r(y)$.
\end{proof}

\begin{lemma}\label{lem:points_added_by_S_have_distance_1_over_2}
 If $x,y\in X_{\kappa^+}$ are distinct points added by $S$, then $d(x,y)\xgeq 1/2$.
\end{lemma}

\begin{proof}
Assume that $x\in T_\alpha$ and $y\in T_\beta$, where $\alpha\xleq\beta$ and
\[
T_\xi=S(A(X_{\xi+1},X_{\xi}),\bigcup_{\delta\xleq\xi}\calA(X_{\delta+1},X_{\delta}))\setminus A(X_{\xi+1},X_{\xi}) 
\]
for $\xi\in\{\alpha,\beta\}$.
The point $y$ is added by $S$, hence by (S2) there exists $Z\in [A(X_{\beta+1},X_{\beta})]^\kappa$ such that
\[
d(y,w)=\frac12+\inf\{d(z,w):z\in Z\} 
\]
for all $w\in T_\beta\setminus \{y\}$.
Since $T_\alpha\subseteq T_\beta$, we have $d(x,y)\xgeq \frac12$.
\end{proof}

\begin{lemma}\label{lem:specjalna_redukcja_dla_operacji_S}
Assume that $x,y\in X_{\kappa^+}$ and $y$ is a weak middle point of $D\in [D_x]^\kappa$.
 Then for all $\epsilon>0$ there exists $D'\in [D]^\kappa$ and its weak middle point $y'$ added by $S$ such that $d(y,y')+d(y',D')<d(y,D)+\epsilon$.
\end{lemma}

\begin{proof}
Fix $\epsilon>0$.
Let $\alpha,\delta_0<\kappa^+$ be such that $r(x)=\alpha+2$ and $r(y)=\delta_0+2$.
From Lemma \ref{lem:operations_A_and_S_do_not_add_middle_points} it follows that $\alpha<\delta_0$.
Define $z_0=y$, $D_0=D$, $r(y)=\delta_0+2$ and assume that there exist $z_0,\ldots,z_n$, $D_n\subseteq \ldots\subseteq D_0$, $\delta_0>\ldots>\delta_n>\alpha$ such that for all $i\xleq n$:
\begin{lista}{(\roman{lista})}
 \item $z_i$ is a weak middle point of $D_i$,
 \item $d(z_{i-1},z_{i})+d(z_{i},D_{i})<d(z_{i-1},D_{i-1})+\epsilon/2^{(i+3)}$ for all $i>0$,
 \item $D_i\in[D_x]^\kappa$,
 \item $z_i\in X_{\delta_i+2}\setminus A(X_{\delta_{i}+1},X_{\delta_{i}})$.
 \end{lista}
 If $z_n\in S(A(X_{\delta_n+1},X_{\delta_n}),\bigcup_{\delta\xleq\delta_n}\calA(X_{\delta+1},X_\delta))$, then we define $y'=z_n$ and $D'=D_n$.
Otherwise, by Lemma \ref{lem:approximation_of_weak_middle_point} we obtain $\alpha<\delta_{n+1}<\delta_n$, $D_{n+1}\in [D_n]^\kappa$ and a weak middle point
\[
z_{n+1}\in X_{\delta_{n+1}+2}\setminus A(X_{\delta_{n+1}+1},X_{\delta_{n+1}})
\]
of $D_{n+1}$ such that
\[
d(z_n,z_{n+1})+d(z_{n+1},D_{n+1})<d(z_n,D_n)+\frac1{2^{n+4}}. 
\]
Observe that for some $n<\omega$ the point $z_n$ is added by $S$, since otherwise we would obtain an infinite decreasing sequence of ordinal numbers.
The following estimate completes the proof:
\begin{multline*}
 d(z_0,z_n)+d(z_n,D_n)\xleq \sum_{i=1}^nd(z_{i-1},z_i)+d(z_n,D_n)\xleq\\
 \sum_{i=1}^{n-1}d(z_{i-1},z_i)+d(z_{n-1},z_n)+d(z_n,D_n)<\\
 \sum_{i=1}^{n-1}d(z_{i-1},z_i)+d(z_{n-1},D_{n-1})+\frac{\epsilon}{2^{n+3}}<\ldots\\
 \ldots<d(z_0,D_0)+\sum_{i=1}^n\frac \epsilon{2^{i+3}}<d(z_0,D_0)+\frac \epsilon4.
\end{multline*}
\end{proof}

\begin{lemma}\label{lem:D_x_is_hereditarily_without_middle_points}
For every $x\in X_{\kappa^+}$ subspace $D_x$ is hereditarily without middle points in $X_{\kappa^+}$.
\end{lemma}

\begin{proof}
Suppose that there exists $x\in X_{\kappa^+}$ such that $D_x$ is not hereditarily without middle points in $X_{\kappa^+}$.
Thus there exist $\alpha<\kappa^+$ and $D\in[D_x]^\kappa$ with a middle point $z\in X_{\kappa^+}$ such that $x\in X_{\alpha+2}\setminus X_{\alpha+1}$.
It follows from Lemma \ref{lem:specjalna_redukcja_dla_operacji_S} that  there exists $D'\in [D]^\kappa$ and its weak middle point $z'$ added by $S$ such that
\begin{equation}\label{eq:epsilon}
d(z,z')+d(z',D')<\frac34. 
\end{equation}
By Lemma \ref{lem:operations_A_and_S_do_not_add_middle_points}, $z'$ is not a middle point of $D'$, hence $d(z',D')>1/2$.
Define $\epsilon=d(z',D')-\frac12$.
Since $z$ is a middle point of $D$, it is also a middle point of  $D'\subseteq D$.
By Lemma \ref{lem:specjalna_redukcja_dla_operacji_S} for $D'$ and its middle point $z$, we obtain a weak middle point  $z''$ of some $D''\in[D']^\kappa$ such that $z''$ is added by $S$ and
\begin{equation}\label{eqn:D''}
d(z,z'')+d(z'',D'')<\frac12+\epsilon. 
\end{equation}
Thus
\[
 d(z,z'')<\frac12+\epsilon-d(z'',D'')=d(z',D')-d(z'',D'')<\frac34-\frac12=\frac14.
\]
By (\ref{eq:epsilon}),
\[
 d(z,z')<\frac34-d(z',D')< \frac34-\frac12=\frac14.
\]
Thus $d(z',z'')\xleq d(z',z)+d(z,z'')<\frac12$.
By (\ref{eqn:D''}), we have
\[
d(z'',D'')<d(z',D')=d(z',D''), 
\]
 hence $z'\neq z''$.
Points $z'$ and $z''$ are added by $S$, hence it follows from Lemma \ref{lem:points_added_by_S_have_distance_1_over_2} that $d(z',z'')\xgeq 1/2$, contrary to $d(z',z'')<1/2$.
\end{proof}

\begin{lemma}\label{lem:computing_distance_with_subset_xleq_kappa}
For every $x,t\in X_{\kappa^+}$, there is $Z\in [D_x]^{\xleq\kappa}$ such that $d(t,y)=d(t,Z,y)$ for all $y\in D_x$.
\end{lemma}

\begin{proof}
Fix $x\in X_{\kappa^+}$ and let $\alpha<\kappa^+$ be such that $x\in X_{\alpha+2}\setminus X_{\alpha+1}$.
Then $d(t,y)=d(t,x)+d(x,y)=d(t,\{x\},y)$ for all $y\in D_x$ and $t\in A(X_{\alpha+2},X_{\alpha+1})\setminus D_x$.
If $t\in D_x$, then $d(t,y)=d(t,\{t\},y)$ for all $y\in D_x$.

Fix $\beta\xgeq\alpha$ and assume that for all $t\in A(X_{\beta+2},X_{\beta+1})$ there exists $Z_t\in [D_x]^{\xleq\kappa}$ such that $d(t,y)=d(t,Z_t,y)$ for all $y\in D_x$.
Fix $t\in X_{\beta+3}\setminus X_{\beta+2}$ added by $S$.
There exists a discrete subspace $Z\in [A(X_{\beta+2},X_{\beta+1})]^{\kappa}$ such that $d(t,y)=d(t,Z,y)$ for all $y\in A(X_{\beta+2},X_{\beta+1})$.
From induction hypothesis, for all $z\in Z$ there exists $T_z\in [D_x]^{\xleq\kappa}$ such that $d(z,y)=d(z,T_z,y)$ for all $y\in D_x$.
Then $\bigcup_{z\in Z}T_z\in [D_x]^{\xleq\kappa}$.
Suppose that there exists $y\in D_x$ such that
\[\textstyle
\epsilon=d(t,\bigcup_{z\in Z}T_z,y)-d(t,y)>0.
\]
There exists $z\in Z$ such that $d(t,z)+d(z,y)<d(t,y)+\epsilon/2$ and
 there exists $w\in T_z$ such that $d(z,w)+d(w,y)<d(z,y)+\epsilon/2$.
Then
\begin{multline*}
d(t,w)+d(w,y)\xleq d(t,z)+d(z,w)+d(w,y)<d(x,z)+d(z,y)+\epsilon/2<\\
\textstyle d(t,y)+\epsilon=d(t,\bigcup_{z\in Z}T_z,y)\xleq d(t,w)+d(w,y),
\end{multline*}
a contradiction.
The case of $t$ being added by $F$ or $A$ is analogous.
\end{proof}

\begin{lemma}
 If $x,y\in X_{\kappa^+}$ are distinct, then $|D_x\cap D_y|\xleq 1$.
\end{lemma}

\begin{proof}
 Fix $x,y\in X_{\kappa^+}$ such that $x\neq y$.
 By (P5), there exist $\alpha,\beta<\kappa^+$ such that  $x\in X_{\alpha+1}\setminus X_{\alpha}$, $y\in X_{\beta+1}\setminus X_\beta$, $D_x\in\calA(X_{\alpha+1},X_\alpha)$ and $D_y\in\calA(X_{\beta+1},X_\beta)$.
 
 Consider the case $r(x)=r(y)$.
 By (A2), we have $D_x\cap D_y=\emptyset$.
 
 Consider the case $r(x)<r(y)$.
 Then $\alpha+2=r(x)+1\xleq r(y)=\beta+1$, hence $X_{\alpha+1}\subseteq A(X_{\alpha+1},X_\alpha)\subseteq X_{\beta+1}$. Therefore $D_x\cap D_y\subseteq X_{\beta+1}\cap D_y$.
By (A1), we have $D_y\cap X_{\beta+1}=\{y\}$.
\end{proof}

\begin{lemma}\label{lem:every_hwmp_space_is_a_subset_of_some_D_x}
Assume that $Y\subseteq X_{\kappa^+}$ is hereditarily without middle points in $X_{\kappa^+}$ and $|Y|>\kappa^+$.
Then there exists $x\in X_{\kappa^+}$ such that $Y\subseteq D_x$.
\end{lemma}

\begin{proof}
Consider the case that there exists $x\in X_{\kappa^+}$ such that $|D_x\cap Y|>\kappa$.
Suppose that there exists $y\in Y\setminus D_x$.
By Lemma \ref{lem:computing_distance_with_subset_xleq_kappa} there exists $Z\in[D_x]^{\xleq\kappa}$ such that $d(y,t)=d(y,Z,t)$ for all $t\in D_x$.
Since $|Z|<|D_x\cap Y|$, there exists $t\in D_x\cap Y\setminus Z$.
Then $y\neq t$, hence
\[
 1=d(y,t)=d(y,Z,t)=\inf\{d(y,z)+1:z\in Z\}.
\]
Since $y\notin Z$, there exist $z,z'\in Z$ such that $z\neq z'$ and $d(x,z),d(x,z')<1/2$.
Then $1=d(z,z')\xleq d(z,x)+d(x,z')<1$, a contradiction.

Consider the case $|D_x\cap Y|\xleq\kappa$ for all $x\in X_{\kappa^+}$.
Define
\[\textstyle
\calK=\{D_x:|D_x\cap Y|=\kappa\}. 
\]
Suppose that $|Y\setminus \bigcup\calK|\xgeq \kappa$.
There exists $T\in [Y\setminus \bigcup\calK]^{\kappa}$.
Thus $|D_x\cap T|<\kappa$ for all $x\in X_{\kappa^+}$.
Since $|T|=\kappa$, there exists $\beta<\kappa^+$ such that $T\subseteq X_\beta$.
From (S1) it follows that there exists a middle point  $t\in X_{\beta+2}$ of $T$, contrary to the fact that $T\in [Y]^\kappa$ and $Y$ is hereditarily without middle points in $X_{\kappa^+}$.
Thus $|Y\setminus\bigcup\calK|<\kappa$.
Then $|Y\cap\bigcup\calK|=|Y|$.
If $|\calK|\xleq\kappa^+$, then
\[\textstyle
\kappa^+<|Y\cap\bigcup\calK|=|\bigcup_{D_x\in\calK}Y\cap D_x|\xleq |\calK|\cdot\sup\{|Y\cap D_x|:D_x\in\calK\}\xleq\kappa^+,
\]
a contradiction.
Thus $|\calK|>\kappa^+$.

Suppose that there exist $x,y\in X_{\kappa^+}$ such that $r(x)=r(y)$ and  $D_x,D_y\in\calK$.
Since $|D_x\cap Y|=|D_y\cap Y|=\kappa$, there exist $x'\in D_x\cap Y\setminus \{x\}$ and $y'\in D_y\cap Y\setminus\{y\}$.
It follows from (A3) that
\[
d(x',y')=d(x',x)+d(x,y')=1+d(x,y)+d(y,y')=2+d(x,y),
\]
contrary to $d(x',y')=1$.

Thus $|\calK\cap \{D_x:r(x)=\alpha\}|\xleq 1$ for all $\alpha<\kappa^+$.
Then
\[\textstyle
  |\calK|=|\bigcup_{\alpha<\kappa^+}\{D_x\in\calK:r(x)=\alpha\}|\xleq\kappa^+,
\]
contrary to $|\calK|>\kappa^+$.
\end{proof}

\begin{lemma}\label{lem:D_x_is_maximal}
 For all $x\in X_{\kappa^+}$, if $D$ is hereditarily without middle points in $X_{\kappa^+}$ and $x\in D$, then $|D|\xleq |D_x|$.
\end{lemma}

\begin{proof}
Fix $x\in X_{\kappa^+}$.
There exists $\alpha<\kappa^+$ such that $x\in X_{\alpha+2}\setminus X_{\alpha+1}$.
Fix $D$ hereditarily without middle points in $X_{\kappa^+}$ such that $x\in D$.
If $|D|\xleq\kappa^+$, then $|D|\xleq |D_x|$, since, by (A1), $|D_x|>\kappa^+$.
Assume that $|D_x|>\kappa^+$.
It follows from Lemma \ref{lem:every_hwmp_space_is_a_subset_of_some_D_x} that there exists $y\in X_{\kappa^+}$ such that $D\subseteq D_y$.
There exists $\beta<\kappa^+$ such that $y\in X_{\beta+2}\setminus X_{\beta+1}$.
If $\beta<\alpha$, then $\beta+3\xleq\alpha+2$ and $D_y\subseteq X_{\beta+3}\subseteq X_{\alpha+2}$, hence $|D_y|\xleq |X_{\beta+3}|\xleq|X_{\alpha+2}|<|D_x|$.

If $\beta=\alpha$, then $x,y\in D_y\cap X_{\beta+2}$ and, by (A1), $D_y\cap X_{\beta+2}=\{y\}$, hence  $x=y$.

If $\beta>\alpha$, then $X_{\alpha+2}\subseteq X_{\beta+1}$.
Thus $x\in D_y\cap X_{\beta+1}=\emptyset$, a contradiction.
\end{proof}

\begin{lemma}\label{lem:different_tau}
 If $x$ and $y$ are distinct points of $X_{\kappa^+}$, then $\tau_\kappa(x,X_{\kappa^+})\neq \tau_\kappa(y,X_{\kappa^+})$.
\end{lemma}

\begin{proof}
Fix $x,y\in X_{\kappa^+}$ such that $x\neq y$.
It follows from Lemma \ref{lem:D_x_is_hereditarily_without_middle_points} and Lemma \ref{lem:D_x_is_maximal} that $\tau_\kappa(x,X_{\kappa^+})=|D_x|$ and $\tau_\kappa(y,X_{\kappa^+})=|D_y|$.
If $r(x)=r(y)$, then, by (A2), we have $|D_x|\neq |D_y|$.
If $r(x)<r(y)$, then, by (A1), we have $|D_x|<|D_y|$.
\end{proof}

Finally, we obtain the main theorem of this paper.

\begin{theorem}
 $X_{\kappa^+}$ is a rigid $\kappa$-superuniversal metric space.
\end{theorem}

\begin{proof}
Suppose that there exists an isometry $f:X_{\kappa^+}\to X_{\kappa^+}$ such that $f(x)\neq x$ for some $x\in X_{\kappa^+}$.
It follows from Lemma \ref{lem:isometry_preserves_tau} that  $\tau_\kappa(x,X_{\kappa^+})=\tau_\kappa(f(x),X_{\kappa^+})$.
Lemma \ref{lem:different_tau} gives $\tau_\kappa(x,X_{\kappa^+})\neq \tau_\kappa(f(x),X_{\kappa^+})$, a contradiction.
Thus $X_{\kappa^+}$ is rigid.

Theorem \ref{thm:X_kappa_is_kappa-superuniversal} says that $X_{\kappa^+}$ is $\kappa$-superuniversal.
\end{proof}

\begin{corollary}
 If $\lambda>\kappa$ is a strongly inaccessible cardinal, then there exists a $\kappa$-superuniversal rigid metric space of cardinality $\lambda$.
\end{corollary}

\begin{proof}
Observe that for a metric space $X$, a family $\calG\subseteq\calD(X)$ and a subset $Y\subseteq X$ we have
\[
|F(X)|\xleq \sum_{\mu<\kappa}(\continuum\cdot|X|)^\mu\xleq 2^\kappa\sum_{\mu<\kappa}|X|^\mu\xleq 2^{\kappa\cdot|X|}, 
\]
\[
|S(X,\calG)|\xleq |X|+|[X]^\kappa|\xleq 2^{\kappa\cdot|X|},
\]
\[
 |A(X,Y)|\xleq \aleph_{\kappa+|X|\cdot 3}.
\]
Fix $\alpha<\lambda$ and assume that $|X_\beta|<\lambda$ for every $\beta<\alpha$.
If $\alpha$ is limit, then $|\bigcup_{\beta<\alpha}X_\beta|<\lambda$, since $\lambda$ is regular.
If $\alpha=\beta+2$, then
\[
|A(X_{\beta+1},X_{\beta})|\xleq \aleph_{\kappa+|X_{\beta+1}|\cdot 3}<\lambda 
\]
\[
 |S(A(X_{\beta+1},X_{\beta}),\bigcup_{\gamma\xleq \beta}\calA(X_{\gamma+1},X_\gamma))|\xleq 2^{\kappa\cdot |A(X_{\beta+1},X_{\beta})|}<\lambda.
\]
Analogously $|F(S(A(X_{\beta+1},X_{\beta}),\bigcup_{\gamma\xleq \beta}\calA(X_{\gamma+1},X_\gamma)))|<\lambda$.
Thus $|X_\lambda|=\lambda$.
\end{proof}

\end{document}